\documentclass[10pt]{amsart}
\usepackage{amsmath,amsthm}
\usepackage[colorlinks=true,linkcolor=blue,urlcolor=blue,citecolor=blue]{hyperref}
\usepackage[all]{xy}

\usepackage[latin1]{inputenc}
\usepackage{amssymb, verbatim}
\usepackage{mathrsfs}

\newcommand     {\B}    {{\mathcal B}}
\newcommand     {\Q}    {{\mathcal Q}}

\newcommand     {\kk}    {{\mathbf k}}

\newcommand     {\Cset}    {{\mathbb C}}
\newcommand     {\Rset}    {{\mathbb R}}

\newcommand     {\Zset}    {{\mathbb Z}}

\newcommand     {\Tr}   {\mbox{Tr}}

\newcommand     {\id}   {\mbox{id}}

\newcommand		{\Galg}	{G_{\Cset}}

\newtheorem{thm}{Theorem}[section]

\newtheorem{lemma}[thm]{Lemma}
\newtheorem{prop}[thm]{Proposition}

\theoremstyle{definition}
\newtheorem{defn}[thm]{Definition}
\newtheorem{conj}[thm]{Conjecture}

\newtheorem{ex}[thm]{Example}

\numberwithin{equation}{section}

\DeclareMathOperator{\Rep}{\mathrm{Rep}}

\DeclareMathOperator{\Span}{\mathrm{Span}}
\DeclareMathOperator{\Spec}{\mathrm{Spec}}

\begin{document}

\title[Connectedness and irreducibility of CQGs]{Connectedness and irreducibility\\ of compact quantum groups}
\author[A.~D'Andrea]{Alessandro D'Andrea}
\address{Dipartimento di Matematica ``Guido Castelnuovo'', Sapienza Universit\`a di
Roma\\ P.le Aldo Moro 5, 00185 Roma, Italy}
\email{dandrea@mat.uniroma1.it}
%\thanks{The author was partially supported by PRIN ``Spazi di Moduli e Teoria di Lie'' fundings from MIUR and project MRTN-CT 2003-505078 ``LieGrits'' of the European Union}
%\date{March 22, 2007}

\author[C.~Pinzari]{Claudia Pinzari}
\address{Dipartimento di Matematica ``Guido Castelnuovo'', Sapienza Universit\`a di
Roma\\ P.le Aldo Moro 5, 00185 Roma, Italy}
\email{pinzari@mat.uniroma1.it}

\author[S.~Rossi]{Stefano Rossi}
\address{Dipartimento di Matematica, Università di Roma Tor Vergata, Via della Ricerca Scientifica 1, 00133 Roma, Italy}
\email{rossis@mat.uniroma2.it}

\keywords{Compact quantum group, connectedness, irreducibility.}
\subjclass[2010]{81R50, 81R60}

\begin{abstract}
We show that a natural notion of irreducibility implies connectedness in the Compact Quantum Group setting. We also investigate the converse implication and show it is related to Kaplansky's conjectures on group algebras.
\end{abstract}

\maketitle

%\tableofcontents

\section{Introduction}

%It is widely accepted that a completely satisfactory theory for general locally compact quantum groups is still missing. 
%Quite the opposite, the theory of genuinely compact quantum groups 
  
Among locally compact quantum groups, whose general theory is admittedly far from complete,  compact quantum groups provide a felicitous class of examples for which a satisfactory theory does exist. This is particularly so when
their representations are looked at. Indeed, the category of finite-dimensional representations of a  given compact quantum group, already at the classical level, displays so rich a structure as to embody virtually any information on the group itself.
To name but few important topological aspects, connectedness \cite{Wang}, local disconnectedness \cite{CDPR}, and topological dimension \cite{DPR} are all properties that the representation category keeps track of very precisely.
Algebraic properties of the group may also be recast in terms of the corresponding category. Notably, the notion of subgroup and its normality, as well as homomorphic images are a case in point.

Analogies between classical and quantum compact groups, however, are far too many to be mentioned at all.  
Even so, cocommutative quantum groups lend themselves to a  more immediate grasp. For instance,
as (classical) Abelian compact groups  only feature irreducible representations of dimension one, so the irreducible
representations of a cocommutative compact quantum group are all still one-dimensional. Furthermore,
the tensor structure of the category corresponds to the group structure of the  dual group, and the conjugate of an
irreducible representation is but its inverse.  
Many of the above topological properties may then be translated into the algebraic language of discrete groups, thought of as dual objects, thus leading back to important yet difficult long-standing conjectures that have risen from group theory over the years, as we shall see in the subsequent sections. Gromov's characterization of groups of polynomial growth, counterexamples to the Burnside problem, Kaplansky's conjectures on group algebras, and the search for a description of groups having Noetherian group algebras may all be interpreted as special cases of geometric and topological issues in compact quantum group theory.\\

In this respect, this paper aims to shed some light on the relation between irreducibility, understood in its algebraic geometrical sense, and connectedness for compact quantum groups. We show that irreducibility always implies connectedness, whereas the inverse implication is at least as hard as tackling Kaplansky's conjecture on the absence of zero divisors in group algebras of torsionless groups. Needless to say, in the classical case connectedness and  irreducibility are equivalent notions, due to the one-to-one correspondence between compact Lie groups and their complexification into reductive groups. Notice that Hopf algebras which are domains \cite{domain}, or more generally prime rings \cite{prime}, are commonly considered for classification purposes.

%%%%Cose da dire: la classe di quantum group compatti, tra i localmente compatti, si presta particolarmente bene a descrivere propriet\`{a} topologico-geometrico-algebrico con un linguaggio rappresentazion-teoretico. Fare esempi con connessione, locale sconnessione, normalit\`{a}, dimensione topologica.
%%% Cos\`{i} come gruppi compatti abeliani hanno rappresentazioni irriducibili solo di dimensione $1$, allo stesso modo nei CQG cocommutativi gli irriducibili hanno dimensione $1$, il prodotto tensoriale  \`{e} un prodotto associativo, e il duale fornisce l'inverso di ciascun irriducibile.
%%Molte delle riformulazioni elencate hanno allora una traduzione gruppale nel contesto dei CGQ cocommutativi, e sono legate a problemi e congetture importanti e difficili.
%In questo articolo, affrontiamo il rapporto tra irriducibilit\`{a} (nel senso geometro-algebrico) e connessione dei quantum group compatti, mostrando come l'irriducibilit\`{a} implichi la connessione, mentre il viceversa conduce a problemi gruppali aperti da decenni. Come sempre, il caso classico \`{e} ben noto. 

\section{Compact quantum groups}

The notion of compact quantum group (CQG) in the $C^*$-algebra formalism  has been developed by Woronowicz \cite{WoroCMP}, see also \cite{MaesVanDaele,NTbook} for a thorough account of the subject. A {\em compact quantum group} is a pair $G=(Q, \Delta)$ where $Q$ is a unital $C^*$-algebra, whose unit is denoted by $1$,  and $\Delta$ is a coassociative unital $^*$-homomorphism
\[\Delta: Q\to Q\otimes Q\]
such that the products $(1\otimes Q)\Delta(Q)$ and $(Q\otimes 1)\Delta(Q)$ are dense in the minimal tensor product $Q\otimes Q$, where $1\otimes Q:=\{1\otimes x: x\in Q\}$ and $Q\otimes 1:=\{x\otimes 1: x\in Q\}$.

The basic and motivating example is given by the algebra $C(G)$ of continuous functions on a compact topological group $G$. In this case the tensor product
$C(G)\otimes C(G)$ is isomorphic with $C(G\times G)$ and the natural coproduct is induced by the group operation itself, i.e.,  $\Delta(f)(g,h)\doteq f(gh)$, for every $f\in C(G) $ and $(g,h)\in G\times G$.
More importantly,  every commutative example is of this form. 
%It is customary to keep the same notation $C(G)$ for $Q$ also when $Q$ is no longer commutative; n
Notice that the structure of $Q$ is thought of as dual to that of $G=(Q, \Delta)$, which is the main object of investigation. One will typically describe properties of $Q$ in terms of a language which is better suited to the structure of $G$. For instance, when $Q$ is commutative, we will say that $G$ is {\em classical}.
  
 A finite-dimensional (unitary) {\em representation} of $G$ is defined as a unitary element $u\in{\B}(H)\otimes Q$, where $H$ is a finite dimensional Hilbert space, satisfying $\Delta(u_{\xi, \eta})=\sum_r u_{\xi, e_r}\otimes u_{e_r,\eta}$. Here $u_{\xi, \eta}=(\xi^*\otimes 1) u (\eta\otimes 1)$, where $\xi,\eta\in H$ and $(e_r)$ is an orthonormal basis of $H$, are the {\em matrix coefficients} of $u$.

As far as representation theory is concerned, an analogue of the Peter-Weyl theorem holds for compact quantum groups as well. More precisely,
the Woronowicz density theorem states that the subalgebra $\Q_G$, which is by definition the subalgebra linearly generated by matrix coefficients of representations of $G$, is dense in $Q$ with respect to its $C^*$-norm. 
Unlike the classical case,  though, the subalgebra $\Q_G$ may well fail to bear a (necessarily not complete) unique $C^*$-norm; however, any such 
 norm on $\Q_G$ can be shown to be bounded between the so-called reduced and maximal norms. The quantum group $G$ is {\em coamenable} when the reduced and maximal norm coincide. Phrased differently, the quantum group $G=(Q, \Delta)$ is coamenable when $Q$ is the only $C^*$-completion of $\Q_G$. All classical compact quantum groups are coamenable.
\subsection{Cosemisimplicity}

Representations of a compact quantum group $G$ can be made into a 
 $C^*$-tensor category with conjugates in the sense of, e.g., \cite{NTbook}. Subrepresentations, quotients, conjugates, direct sums, tensor products of representations as well as irreducible representations and intertwiners are defined in the obvious way.
Every representations can be decomposed as a direct sum of irreducible representations, in a unique way up to equivalence. Every irreducible representation is finite dimensional. We denote by $\Rep G$ the corresponding Grothendieck (fusion) ring: this is a $\Zset$-algebra endowed with an involution $*$ induced by taking dual of representations.

We have seen above that the linear span $\Q_G$ of matrix coefficients of representations is a canonical dense $^*$-subalgebra of $Q$. Furthermore, it has the structure of an honest Hopf $^*$-algebra \cite{WoroCMP, cqg}, which is {\em cosemisimple} by the abovementioned complete reducibility; representations of $G$ are the same as $\Q_G$-comodules.

Due to cosemisimplicity, $\Q_G$ has a unique {\em Haar state} $h$, which means that $h$ is a state satisfying the invariance condition $(h\otimes \id)(\Delta(a))=h(a)1=(\id\otimes h)(\Delta(a))$ for all $a\in Q$. 
It is uniquely determined by demanding that $h(1) = 1$ and that it annihilates all coefficients of non-trivial irreducible representations. The Haar state is always {\em positive} in the compact quantum group setting, which means that $h(a^* a) > 0$ for all $0 \neq a \in \Q_G$.
Quite remarkably, the Haar state is uniquely determined at the $C^*$-algebraic level of $Q$ too \cite{WoroCMP}, by only requiring
that it satisfies the invariance condition.

\subsection{Character theory}
In the classical theory,  with any finite-dimensional representation $\pi$ of a compact group $G$ one may associate 
its character $\chi^\pi$, which is the continuous function $\chi^\pi(g)\doteq {\rm Tr}(\pi(g))$, $g\in G$. 
If now $\{\pi^i, i\in I\}$ is a complete family of inequivalent irreducible representations of $G$, the set of corresponding characters $\{\chi^i, i\in I\}$
is an orthonormal system in the Hilbert space $L^2(G, {\rm m})$, where ${\rm m}$ is the Haar measure of
$G$. Unless $G$ is Abelian, however, the functions thus obtained will fail to be an orthonormal basis of $L^2(G, {\rm m})$. In fact, they are a basis for the Hilbert subspace $ZL^2(G, {\rm m})\subset L^2(G, {\rm m})$,
which is the closure in $L^2(G, {\rm m})$ of the Banach space $ZC(G):=\{f\in C(G):\, f(gh)=f(hg), \,\textrm{for every}\, g,h\in G\}$ of {\em central functions}. 

One of the features of compact quantum groups is that most of the usual character theory outlined above extends to the quantum setting. If a finite-dimensional representation is described by the unitary element $u \in{\B}(H)\otimes Q$, the corresponding character can be still defined as $\chi(u) = (\Tr \otimes \id)(u) \in \Q_G$. Associating with each finite-dimensional representation of $G$ its character sets up a ring homomorphism $\chi: \Rep G \to \Q_G$ which commutes with the corresponding $*$-involutions. Characters $\{\chi^i, i \in I\}$ corresponding to a complete family of pairwise non-isomorphic irreducible representations of $G$ continue to satisfy the usual orthonormality relations $h(\chi^i (\chi^j)^*) = \delta^{ij}$, showing that $\chi: \Rep G \to \Q_G$ is indeed injective, thus providing an embedding of $\Rep G$ inside $\Q_G$. %At least in the Kac algebra case, it is known that the image of $\Rep G$ via $\chi$ coincides with the subalgebra of {\em central functions}.

To the best of our knowledge, in a general quantum framework the subspace $\Span_\Cset\langle\chi^i, i\in I\rangle$ is no longer known to be dense in $Z\Q:=\{x\in \Q\,|\, \Delta(x)=\theta(\Delta(x))\}$, where $\theta$ is the $^*$-isomorphism of $\Q\otimes \Q$ given by $\theta(x\otimes y)=y\otimes x$, $x,y\in\Q$, although this is certainly the case for Abelian compact quantum groups.

\subsection{Finite quantum groups and connectedness}

Wang extended in \cite{Wang} the notion of connectedness to the compact quantum setting: a compact quantum group $G = (Q, \Delta)$ is {\em connected} if the only finite-dimensional unital Hopf $\,^*$-subalgebra of $Q$ is the base field $\Cset$; using the dual group language, the only quotient quantum groups of $G$ are {\em finite}. One may reformulate this notion in terms of representation theory, and one of the results from \cite{CDPR} shows that $G$ is connected precisely when each nontrivial irreducible representation $u$ of $G$ requires infinitely many pairwise non-isomorphic irreducible summands to decompose all tensor powers of $u \oplus u^*$; in other words, the only {\em torsion} irreducible representation is the trivial one.

In the classical setting, if a compact topological group $G$ is not connected, then the corresponding group of connected components $G/G^\circ$ is totally disconnected and thus has nontrivial finite quotients; then each nontrivial irreducible representation of such quotients lifts to a torsion representation of $G$; the viceversa clearly holds by standard Lie representation theory.

\subsection{Finite quantum groups and semisimplicity}

When $\Q_G$ is finite dimensional, i.e., when $G$ is {\em finite}, then $\Q_G$ coincides with its completion $Q$, which is a finite-dimensional $C^*$-algebra. It is well known \cite{Takesaki} that finite-dimensional $C^*$-algebras are semisimple. A complex semisimple algebra is always a direct sum of complex matrix algebras; as the counit $\epsilon: \Q_G \to \Cset$ is a surjective ring homomorphism, one of this direct summand is isomorphic to $\Cset$.

Summing up, if $G$ is a finite quantum group, then $\Q_G$, viewed as a $*$-algebra, is isomorphic to a finite sum of complex matrix algebras, at least one of the summands being isomorphic to $\Cset$.

\subsection{Connectedness and irreducibility of complex algebraic groups}

If $P_0 \neq P_1$ are points in a compact Hausdorff space $X$, choose disjoint open neighbourhoods $U_{i}\ni P_i$. By Urysohn's Lemma, one finds continuous functions $f_i: X \to \Rset$ whose value is $1$ on $P_i$ and vanish on the complement of $U_{i}$. Then $f_0 \cdot f_1 = 0$, yet neither factor is the constant zero function, thus showing that the $C^*$-algebra $C(X)$ is never a domain.

As a consequence, if $G\neq \{1\}$ is a compact Lie group, then $Q = C(G)$ has zero divisors. However, as soon as $G$ is connected, then $\Q_G$ is an integral domain. Indeed, $\Q_G$ can be understood as the coordinate ring of the affine complex algebraic group $\Galg = \Spec(\Q_G)$ admitting $G$
as a maximal compact subgroup. It is well known that an algebraic group is connected if and only if it is irreducible. Indeed, the intersection of irreducible components is singular, and the singular part of an algebraic group, which is a Zariski closed set, must be empty as it is invariant under all left multiplications. Irreducibility of $\Galg$ then translates into $\Q_G$ being an integral domain. 

The main goal of the present note is to show that a suitably generalized notion of irreducibility implies connectedness for all (even non-classical) compact quantum groups. More explicitly, we will show in Proposition \ref{irredimpliesconnected} below that a compact quantum group $G$ is connected as soon as the corresponding canonical Hopf $*$-algebra $\Q_G$ is a (possibly noncommutative) domain.

\section{Discrete groups}

\subsection{Abelian compact quantum groups}
An important class of non-classical examples is provided by {\em Abelian} compact quantum groups, which correspond to cocommutative instances of $C(G)$. If $\Gamma$ is a discrete (ordinary) group then the group $C^*$--algebra
$C^*(\Gamma)$, which is the completion of the group algebra ${\mathbb C}\Gamma$ in the maximal $C^*$--norm, becomes a compact quantum group with coproduct $\Delta(\gamma)=\gamma\otimes\gamma$, $\gamma\in \Gamma$. We may also consider the reduced $C^*$--completion $C^*_{\text{red}}(\Gamma)$ and still obtain a compact quantum group. These are cocommutative examples and every cocommutative compact quantum group can be obtained as the completion of ${\mathbb C}\Gamma$ with respect to some $C^*$--norm, which is bounded between the reduced and the maximal norm. 
 An Abelian compact quantum group $C^*(\Gamma$) is coamenable if and only if $\Gamma$ is amenable as a group.

The correspondence between Abelian CQG and discrete groups provides a bridge associating topological properties of compact quantum groups with structural aspects of discrete groups.

\subsection{Topology of CQGs and structure of discrete groups}

The notion of topological (Lebesgue) dimension of a compact topological group $G$ is related to the Gelfand-Kirillov dimension of $\Q_G$, which has been rephrased in representation theoretic terms in \cite{DPR}. This can be used to extend the concept of topological dimension to all (possibly non-classical) compact quantum groups. In the special case of an Abelian compact quantum group $G = C^* \Gamma$, the topological dimension of $G$ is only finite when $\Gamma$ is a group of polynomial growth, in which case it equals the growth degree. A celebrated result by Gromov characterizes all finitely generated groups of polynomial growth.
\begin{thm}[Gromov]
A finitely generated group has polynomial growth if and only if it is virtually nilpotent.
\end{thm}

When $G = C^* \Gamma$ is a connected Abelian CQG, irreducible representations of $G$ are in one-to-one correspondence with elements of $\Gamma$; then tensor product of irreducible representations is given by the group multiplication, and the dual by taking the inverse element. We have seen above that connectedness of a CQG can be rephrased in terms of the absence of torsion representations; indeed, $G = C^* \Gamma$ is connected if and only if $\Gamma$ is torsionless. A long-standing conjecture for group algebras of discrete groups is the following:
\begin{conj}[Kaplansky]\label{Kaplansky}
Let $\kk$ be a field. A discrete group $\Gamma$ is torsionless if and only if its group algebra $\kk \Gamma$ has no zero-divisors.
\end{conj}
Kaplansky's conjecture is known to hold for polycyclic-by-finite groups. As virtually nilpotent groups are of this type, Kaplansky's conjecture holds for finitely generated groups of polynomial growth. Notice that polycyclic-by-finite groups are the only groups known to yield a Noetherian group algebra.

In a similar---albeit more analytical---fashion, the so-called Kadison-Kaplansky conjecture states that the reduced $C^*$-algebra $C^*_r(\Gamma)$
has no non-trivial projections if $\Gamma$ is a torsion-free discrete group. The conjecture has been proved true for  word-hyperbolic groups \cite{Puschnigg}.

\section{Irreducibility implies connectedness}

\begin{defn}
Let $G$ be a compact quantum group. Then
\begin{itemize}
\item $G$ is {\em irreducible} iff $\Q_G$ is a domain;
\item $G$ is {\em weakly irreducible} iff $\Rep G$ is a domain;
\item $G$ is {\em connected} iff it has no nontrivial torsion irreducible representation.
\end{itemize}
\end{defn}

\begin{ex}
Let $G$ be a semisimple compact Lie group. We have seen above that $G$ can be viewed as maximal compact subgroup of the semisimple complex algebraic group $\Galg = \Spec \Q_G$. Representations of $\Galg$ restrict to representations of $G$, and irreducibility, tensor product, direct sums are preserved.

If we denote by $r$ the rank of $\Galg$ and by $u_1, \dots, u_r$ its irreducible fundamental representations, then mapping $x_i$ to $u_i$ provides an isomorphism $\Zset[x_1, \dots, x_r] \simeq \Rep \Galg = \Rep G$, thus showing that $\Rep G$ is a commutative domain, hence $G$ is a weakly irreducible CQG.
\end{ex}

Connectedness,  irreducibility and
weak irreducibility are all preserved under taking inverse limits:
\begin{lemma}\label{limit}
Connectedness, irreducibility and weak irreducibility are preserved by inverse limits of CQGs.
\end{lemma}
\begin{proof}
If $G$ is the inverse limit of its quotients $G_i$, then $\Q_G$ is by definition the direct limit of $\Q_{G_i}$. 
If now each $G_i$ is connected then no finite-dimensional Hopf subalgebra can be contained in any of the Hopf algebras $\Q_{G_i}$, hence
no finite-dimensional Hopf subalgebra can be contained in $\Q_G$ either. The conclusion is now easily reached, for
any finite-dimensional Hopf subalgebra of the $C^*$-algebra $Q$ is actually contained in $\Q_G$.  

As for irreducibility, if all $\Q_{G_i}$ are domains, then also $\Q_G$ is a domain, as every pair of nonzero elements $x, y\in \Q_G$ satisying $xy = 0$ must be contained in $\Q_{G_i}$ for some $i \in I$. The same argument holds for $\Rep G$.
\end{proof}

\begin{prop}\label{domainimpliesdomain}
Let $G$ be a compact quantum group. If $\Q_G$ is a domain, then $\Rep G$ is a domain.
\end{prop}
\begin{proof}
The character function $\chi: \Rep G \to \Q_G$ is an injective ring homomorphism. Then if $\Q_G$ is a domain, $\Rep G \simeq \chi(\Rep G) \subset \Q_G$ must be a domain too.
\end{proof}

\begin{prop}\label{quasiimpliesconnected}
Let $G$ be a compact quantum group. If $\Rep G$ is a domain, then $G$ is connected.
\end{prop}
\begin{proof}
If $G$ is not connected, then it has a finite (i.e., finite dimensional as a $\Cset$-vector space) quotient $H$. Then the fusion ring $A: = \Rep H$, which sits inside $R: = \Rep G$, has finitely many irreducibles and is closed under direct sum, tensor product and subobjects. The ring $A$ is a finitely generated free $\Zset$-module, as the finitely many irreducible representations of $H$ linearly span it over $\Zset$, so that each of its elements admits a unique (monic) minimal polynomial in $\Zset[x]$.

Assume $R$, hence $A$, to be a domain. If $u$ is a nontrivial irreducible representation of $H$, its minimal polynomial $p(x) \in \Zset[x]$ is irreducible of degree $> 1$, hence it certainly has no roots in $\Zset$. However, the dimension function $\dim: R \to \Zset$ is a ring homomorphism, so that $\dim u \in \Zset$ is a root of $p$, thus yielding a contradiction.
\end{proof}

\begin{prop}\label{irredimpliesconnected}
Let $G$ be a compact quantum group. If $\Q_G$ is a domain, then $G$ is connected.
\end{prop}
\begin{proof}
This follows from Propositions \ref{domainimpliesdomain} and \ref{quasiimpliesconnected}, but we also provide a direct proof of a somewhat different flavour.
If $G$ is not connected, then it has a nontrivial finite quotient quantum group $H$. Then $\Q_H \subset \Q_G$ is a semisimple $C^*$-algebra, which is a direct sum of matrix algebras. However, a nontrivial direct sum of matrix algebras is never a domain.
\end{proof}
Notice that when $G$ is a compact topological group which is not connected, then above proofs locates a finite-dimensional $*$-algebra of the commutative algebra $C(G)$, thus yielding a decomposition of $1 \in C(G)$ into a sum of orthogonal central idempotents. This also forces $C(G)$ 
to decompose into a direct sum of finitely many ideals. This fact is likely not to hold in general, and seems to require $G$ to be either classical or finite.
We may summarize Propositions \ref{domainimpliesdomain}-\ref{irredimpliesconnected} in the following
\begin{thm}
Let $G$ be a compact quantum group. Then
\[G \mbox{ irreducible } \implies G \mbox{ weakly irreducible } \implies G \mbox{ connected}.\]
\end{thm}

\section{Does connectedness imply irreducibility?}

It is well known that connectedness is equivalent to irreducibility as soon as $G$ is a compact Lie group, hence for all {\em classical} (i.e., commutative) CQGs by Lemma \ref{limit}. In the compact Lie case, $\Q_G$ is the complex algebra of regular functions of the affine group corresponding to $G$, which is irreducible as soon as $G$ is connected.

In the non-classical case, proving that connectedness implies irreducibility certainly constitutes a difficult problem, already when $G$ is Abelian. Indeed, we have seen that when $G=C^* \Gamma$, then $\Cset\Gamma$ being a domain only forces connectedness of $G$ if Conjecture \ref{Kaplansky} holds. The general statement is thus at least as difficult as proving Kaplansky's conjecture.

However, connectedness and irreducibility are equivalent also when $G$ is a cocommutative compact matrix quantum group $C^* \Gamma$ of finite topological dimension \cite{DPR}, as the polynomial growth requirement forces $\Gamma$ to be virtually nilpotent, whence polycyclic-by-finite, and we have seen above that Kaplansky's conjecture is known to hold for such groups. As having finite topological dimension commutes with inverse limits, Lemma \ref{limit} shows that equivalence of connectedness and irreducibility holds for every Abelian (i.e., cocommutative) compact quantum group of finite topological dimension.

\vfill

\end{document}